\documentclass[11pt]{amsart}
\usepackage{amssymb, mathrsfs, amsmath, color, tikz-cd, verbatim}

\def\L8{L_\infty}

\newtheorem{thm}{Theorem}
\newtheorem{prp}[thm]{Proposition}
\newtheorem{cor}[thm]{Corollary}
\newtheorem{lmm}[thm]{Lemma}
\newtheorem{dfn}[thm]{Definition}

\title[]{Center of mass and K\"ahler structures}

\author[S.~Wilson]{Scott O. Wilson}
  \address{Scott O. Wilson, Department of Mathematics, Queens College, City University of New York, 65-30 Kissena Blvd., Flushing, NY 11367}
  \email{scott.wilson@qc.cuny.edu}

\author[M.~Zeinalian]{Mahmoud~Zeinalian}
  \address{Mahmoud Zeinalian, Department of Mathematics, LIU Post, Long Island University, 720 Northern Boulevard, Brookville, NY 11548, USA} 
  \email{mzeinalian@liu.edu}


\begin{document}

\begin{abstract} 
There is a sequence of positive numbers $\delta_{2n}$, such that for any connected $2n$-dimensional Riemannian manifold $M$, there are two mutually exclusive possibilities:
\begin{enumerate}
\item[]
\item There is a complex structure on $M$ making it into a K\"ahler manifold.
\item For any almost complex structure $J$ compatible with the metric, at every point $p\in M$, there is a smooth loop $\gamma$ at $p$ such that $$dist(J_p, hol_\gamma^{-1}J_phol_\gamma)> \delta_{2n}.$$ 
\item[]
\end{enumerate}

\end{abstract}
\maketitle

\section{Introduction}

K\"ahler manifolds possess a tremendous amount of interesting structure, and therefore have several equivalent characterizations. It has been a focus of much research to determine conditions under which manifolds do (or do not) admit a K\"ahler structure. This short note shows that the holonomy action on the space of almost hermitian structures determines two mutually exclusive cases, according to whether there is a structure that is \emph{nearly preserved} at some point, by proving that any manifold with a nearly preserved almost hermitian structure at some point in fact admits a K\"ahler structure. The novel idea is to use a center of mass argument, averaging a given almost complex structure at a point over the holonomy action, and parallel transporting the result to obtain a global K\"ahler structure.

Due to the plethora of topological implications for having a K\"ahler structure, one may deduce several interesting consequences that do not mention the K\"ahler condition at all. For example, we deduce by \cite{DGMS} that any almost hermitian manifold which has a non-trivial Massey product (or more generally is not formal) also has a holonomy action on the space almost hermitian structures which is bounded away from the identity in a way that is precisely quantifiable.

The authors would like to thank Dennis Sullivan for conversations about this.


\section{Main result}
For a $2n$-dimensional real vector space $V$ with an inner product $g$, let 
\[
\mathbb{J}(V, g)=\{J: V \to V | \, J^2=-id, \, g\left(Ju, Jv \right)=g\left(u, v\right)\}
\] 
denote the space of metric almost complex structures on $V$.
\begin{lmm}
 $\mathbb{J}(V, g)$ is a compact smooth manifold. The tangent space at a point $J$ is  
 \[
 T_J\mathbb{J}(V, g)=\{\phi: V\to V~|~ \phi J=-J\phi \, \, \, \textrm{and} \, \, \,  \phi^\dagger= -\phi\}\]
 and has a bilinear form $\tilde{g}(\phi, \psi)=tr(\phi\psi^\dagger)$, for all $\phi, \psi \in T_J\mathbb{J}(V, g)$, where $\psi^\dagger$ denotes the adjoint of $\psi$. This makes $\mathbb{J}(V, g)$ into a Riemannian manifold on which $O(V)$ acts transitively by isometries. 

\end{lmm}

\begin{proof}  $\mathbb{J}(V, g)$ can be identified as the $O(V)$-homogenous space $O(V)/U(V)$ where $U(V)$ is defined using any fixed $J$ on $V$, and the action of $O(V)$ on  $\mathbb{J}(V, g)$ is given by conjugation.
 Therefore $\mathbb{J}(V, g)$ is a smooth manifold. The induced action of $O(V)$ on $T_J\mathbb{J}(V, g)$ is also given by conjugation, so that trace is invariant.
\end{proof}

\begin{lmm} \label{lmm;WD}
For any two metric vector spaces $(V, g)$ and $(W, h)$ of the same dimension, the Riemannian manifolds $\mathbb{J}(V, g)$ and  $\mathbb{J}(W, h)$  are isometric. In particular, for a vector space $V$ and two metrics $g$ and $g'$, the Riemannian manifolds $\mathbb{J}(V, g)$ and  $\mathbb{J}(V, g')$ are isometric. 
\end{lmm}
\begin{proof} The Gramm-Schmidt process ensures that there is a linear isometry $f: (V, g)\to (W, h)$. Conjugation by this isometry gives the desired Riemannian isometry.
\end{proof}

For an argument below, we will require convex balls in $\mathbb{J}(V, g)$ for which is there is a well defined notion of center of mass.  For any Riemannian manifold $M$, balls of radius $r$ are convex if
\[
r = \textrm{min} \Big\{ \frac{ \textrm{inj} \, M }{2}, \frac{\pi}{ 2 \sqrt \epsilon} \Big\},
\]
where $inj \, M$ denotes the injectivity radius of $M$, and $\epsilon$ is a finite positive upper bound on the sectional curvature (c.f. \cite{Ch} Prop IX.6.1). Also, for any such $\epsilon$, we have the following theorem.

\begin{thm}[Karcher, \cite{K}] \label{thm;K}
Let $f: X \to M$ be a measurable map from a probability space  $(X,m)$ to a Riemannian manifold $M$.
If $f(X)$ is contained in a convex subset $B$ of $M$, with diameter less than or equal to $\pi / 2 \sqrt \epsilon$,
then there is a unique center of mass in $B$, defined by the minimum of 
\[
E(y) = \frac{1}{2} \int_X d^2\left(f(x) , y\right) \, m(x).
\]
\end{thm}

Additionally, the center of mass is natural with respect to isometries (\cite{K}, 1.4.1). 
We refer the reader to \cite{Ch} Prop IX.7.1 for an exposition of center of mass. 

\begin{dfn} \label{dfn;delta2n}
For $n >1$ let 
\[
\delta_{2n} = \textrm{min} \Big\{ \frac{ \textrm{inj} \, \,  \mathbb{J}(V, g) }{2}, \frac{\pi}{ 4 \sqrt \epsilon} \Big\}
\]
where $V$ is a real vector space of dimension $2n$ with any metric $g$. Here we choose, once and for all, a finite positive upper bound $\epsilon$ on the sectional curvature of $\mathbb{J}(V, g)$ at one point, which by homogeneity works for all points. By Lemma \ref{lmm;WD}, $\delta_{2n}$ depends only on the dimension of $V$.
\end{dfn}

An interesting question (that we will not address here) is whether there is a positive lower bound for the set of all least such $\delta_{2n}$, independent of $n$.

\begin{dfn} A complex structure $J\in \mathbb{J}(V, g)$ is said to be \emph{nearly preserved} by a subgroup $H \subset O(V)$ if the orbit $HJ=\{\phi^{-1}J\phi ~|~ \phi \in H\}$ lies inside the ball $B(J, \delta_{2n})\subset  \mathbb{J}(V, g)$ with center $J$ and radius $\delta_{2n}$.
\end{dfn}

Recall that any closed subgroup $H$ of $O(V)$ is a compact Lie group, admitting a bi-invariant Haar measure, which is unique up to a constant. Therefore any such $H$ has a unique probability measure (of total mass equal to one).

\begin{prp} \label{prp;HJ=J}
Let $J\in \mathbb{J}(V, g)$ be nearly preserved by a closed subgroup $H$ of $O(V)$. Then there is a $J'\in \mathbb{J}(V, g)$ such that $HJ'=\{J'\}$.
\end{prp}
\begin{proof}
Consider the orbit $HJ\subset B(J, \delta_{2n}) \subset \mathbb{J}(V, g)$. By assumption, 
$B(J, \delta_{2n})$ is a convex ball about $J$. 
Consider the mapping $H \to HJ \subset \mathbb{J}(V, g)$, from the probability space $H$ onto its orbit. 
By Definition \ref{dfn;delta2n} and Theorem \ref{thm;K}, the set $HJ$ has a unique center of mass $J'$ in $S$.
Since the orbit of the action of $H$ on the set $HJ$ is itself, and the center of mass is unique and natural with respect to isometries given by the $H$-action, $H$ fixes $J'$. 
\end{proof}

Given a point $p\in M$, let $H_p=Im(\rho_p)$ be the image subgroup of the holonomy homomorphism $\rho_p:\Omega_p(M) \to O(T_pV)$, which is known to be closed by the Ambrose-Singer theorem. The following proposition is a standard result which we include for completeness.

\begin{prp} \label{prp;K}
Let $(M, g)$ be a connected Riemannian manifold and assume $J_p\in \mathbb{J}(T_pM, g)$ is invariant under the action of $H_p=Im(\rho_p)$. Then $M$ admits a unique almost complex structure $J: TM\to TM$ agreeing with $J_p$ and making $(M, g , J)$ into a K\"ahler manifold.  
\end{prp}
\begin{proof}
For any $q\in M$ define $J_q: T_q M\to T_qM$ by $P_\lambda \circ J_p \circ P_\lambda^{-1}$, where $P_\lambda: T_pM\to T_qM$ is the Riemannian parallel transport along any smooth path $\lambda$ in $M$ from $p$ to $q$. Since $M$ is connected, such a path $\lambda$ always exists and $J_q$ is independent of the choice of the path because $H_p J_p=\{J_p\}$. Thus, we have defined a smooth complex structure $J: TM\to TM$, and by way of construction, $J$ is compatible with all parallel transports, and therefore $\nabla J=0$. Since the Levi-Cevita connection $\nabla$ is torsion free, and $J$ is still Hermitian, $J$ is integrable; see Corollary 3.5 of \cite{KN}. Therefore, $(M, g, J)$ is a K\"ahler manifold. 
\end{proof}

\begin{cor}
With $(M,g)$ as above, if there is a $p\in M$ and a $J_p \in \mathbb{J}(T_pM, g)$ that is nearly preserved by $H_p=Im(\rho_p)$, then there is a $J': TM\to TM$ making $(M,g,J')$ into a K\"ahler manifold. 
\end{cor}

\begin{proof} This follows from Proposition \ref{prp;HJ=J} and Proposition \ref{prp;K}.
\end{proof}

\begin{thm}
There is a sequence of positive numbers $\delta_{2n}\in \mathbb R$ such that for any connected $2n$-dimensional Riemannian manifold $M$, one of the two following mutually exclusive properties hold:

\begin{enumerate}
\item[]
\item There is a complex structure on $M$ making it into a K\"ahler manifold.
\item For any almost complex structure $J$ compatible with the metric, at every point $p\in M$, there is a smooth loop $\gamma$ at $p$ such that $$dist(J_p, hol_\gamma^{-1}J_phol_\gamma)> \delta_{2n}.$$
\end{enumerate}
\end{thm}

\begin{proof} If $1)$ is true then $2)$ is clearly false, and the converse follows from the previous corollary.
\end{proof}

\bibliographystyle{plain}

\end{document}